\documentclass[12pt,oneside]{amsart}

 \usepackage{amssymb}
\usepackage[margin=1in]{geometry}  
\usepackage{setspace}
\usepackage{ulem}       
\usepackage{amsmath, amsthm, amssymb}
  \usepackage{verbatim} 
\usepackage{xspace}
\usepackage{mathrsfs}
\usepackage[pdftex,bookmarks,colorlinks,breaklinks]{hyperref}  
\hypersetup{linkcolor=blue,citecolor=red,filecolor=dullmagenta,urlcolor=darkblue} 
 \usepackage{graphicx}  
 \usepackage{color}	
 \usepackage{mathabx}   
\newtheorem{theorem}{Theorem}[section]
\newtheorem{corollary}[theorem]{Corollary}
\newtheorem{lemma}[theorem]{Lemma}
\newtheorem{conjecture}[theorem]{Conjecture}

\newtheorem{remark}[theorem]{Remark}

\newcommand{\la}{\lambda}

\newcommand{\al}{\alpha}
\newcommand{\scr}{\mathscr} 
\newcommand\mult{\operatorname{mult}}

      \makeatletter
      \def\@setcopyright{}
      \def\serieslogo@{}
      \makeatother

\begin{document}

   \author{M. A.  Bahmanian}\thanks{Research is partially supported by NFIG at ISU, and NSA Grant H98230-16-1-0304}
   \address{Department of Mathematics,
  Illinois State University, Normal, IL USA 61790-4520}

   \title[Factorizations of Complete Multipartite Hypergraphs]{Factorizations of Complete Multipartite Hypergraphs}

   \begin{abstract}
In a mathematics workshop with $mn$ mathematicians from $n$ different areas, each area consisting of $m$ mathematicians, we want to create a collaboration network. For this purpose, we would like to schedule daily meetings between groups of size three, so that (i) two people of the same area meet one person of another area, (ii) each person has exactly $r$ meeting(s) each day, and (iii) each pair of people of the same area have exactly $\lambda$ meeting(s) with each person of another area by the end of the workshop. Using hypergraph amalgamation-detachment, we prove a more general theorem. In particular we show that above meetings can be scheduled if: $3\divides rm$, $2\divides rnm$ and  $r\divides 3\lambda(n-1)\binom{m}{2}$. This result can be viewed as an analogue of Baranyai's theorem on factorizations of complete multipartite hypergraphs.
   \end{abstract}

   \subjclass[2000]{ 05C70, 05C51, 05C15, 05C65, 05B30}

   \keywords{Baranyai's Theorem, Amalgamations, Detachments, Multipartite Hypergraphs, Factorizations, Decompositions}
   \date{\today}
   \maketitle

\section{Introduction}
Throughout this paper, $\mathbb{N}$ is the set of positive integers,   $m,n,r,\lambda\in \mathbb{N}$, and $[n]:=\{1,\dots,n\}$. In a mathematics workshop with $mn$ mathematicians from $n$ different areas, each area consisting of $m$ mathematicians, we want to create a collaboration network. For this purpose, we would like to schedule daily meetings between groups of size three, so that (i) two people of the same area meet one person of another area, (ii) each person has exactly $r$ meeting(s) each day, and (iii) each pair of people of the same area have exactly $\lambda$ meeting(s) with each person of another area by the end of the workshop. Using hypergraph amalgamation-detachment, we prove a more general theorem. In particular we show that above meetings can be scheduled if: $3\divides rm$, $2\divides rnm$ and  $r\divides 3\lambda(n-1)\binom{m}{2}$. 

A {\it hypergraph} $\mathcal G$ is a pair $(V,E)$ where $V$ is a finite set called the vertex set, $E$ is the edge multiset, where every edge is itself a multi-subset of $V$. This means that not only can an edge  occur multiple times in $E$, but also each vertex can have multiple occurrences within an edge.  
The total number of occurrences of a  vertex $v$ among all edges of $E$ is called the {\it degree}, $d_{\mathcal G}(v)$ of $v$ in $\mathcal G$. For $h\in \mathbb{N}$, $\mathcal G$ is said to be $h$-{\it uniform} if $|e|=h$ for each $e\in E$.   
For  $r, r_1,\dots,r_k\in \mathbb{N}$, an $r$-factor in a hypergraph $\mathcal G$ is a spanning  $r$-regular sub-hypergraph, and an {\it $(r_1,\dots,r_k)$-factorization}  is a partition of the edge set of $\mathcal G$ into $F_1,\dots, F_k$ where $F_i$ is an $r_i$-factor for $i\in [k]$. We abbreviate $(r,\dots,r)$-factorization to $r$-factorization.  

The hypergraph $K_n^h:=(V,\binom{V}{h})$ with $|V|=n$ (by $\binom{V}{h}$ we mean the collection of all $h$-subsets of $V$) is called a  {\it complete} $h$-uniform hypergraph. In connection with Kirkman's schoolgirl problem  \cite{Kirk1847}, Sylvester conjectured that $K_n^h$ is 1-factorable if and only if $h\divides n$. This conjecture was finally settled by Baranyai \cite{Baran75}.   Let $\scr K_{n\times m}^3$ denote the 3-uniform  hypergraph with vertex partition $\{V_i:i\in [n]\}$, so that $V_i=\{x_{ij}:  j\in [m]\}$ for $i \in [n]$, and with edge set $E=\{\{x_{ij}, x_{ij'}, x_{kl}\} : i,k\in [n], j,j',l\in [m], j\neq j',  i\neq k\}$.
One may notice that finding an $r$-factorization  for $ \scr K_{n\times m}^3$  is equivalent to scheduling the meetings between mathematicians with the above restrictions for the case $\lambda=1$. 

If we replace every edge $e$ of $\mathcal G$ by $\lambda$ copies of $e$, then we denote the new hypergraph by $\lambda \mathcal G$.   In this paper, the main result is the  following theorem which is obtained by proving a more general result (see Theorem \ref{moregenres}) using amalgamation-detachment techniques.
 \begin{theorem}\label{type2suff}
$\lambda \scr K_{m\times n}^{3}$ is $(r_1,\ldots,r_k)$-factorable if 
\begin{enumerate}
\item [(S1)] $3\divides r_im$ for $i\in [k]$, 
\item [(S2)] $2\divides r_imn$ for $i\in [k]$,  and 
\item [(S3)] $\sum_{i=1}^k r_i= 3\lambda(n-1)\binom{m}{2}$. 
\end{enumerate}
\end{theorem}
In particular, by letting $r=r_1=\dots=r_k$ in Theorem \ref{type2suff}, we solve the Mathematicians Collaboration Problem in the following case.
\begin{corollary} 
$\lambda \scr K_{m\times n}^{3}$ is $r$-factorable if 
\begin{enumerate}
\item [\textup {(i)}] $3\divides rm$,
\item [\textup {(ii)}] $2\divides rnm$, and 
\item [\textup {(iii)}]  $r\divides 3\lambda(n-1)\binom{m}{2}$. 
\end{enumerate}
\end{corollary}
 The two results above can be seen as analogues of Baranyai's theorem for complete 3-uniform ``multipartite"  hypergraphs. We  note that  in fact, Baranyai \cite{Baran79} solved the problem of factorization of complete uniform multipartite hypergraphs, but here we aim to solve this problem under a different notion of ``multipartite". In Baranyai's definition, an edge can have at most one vertex from each part, but here we allow an edge to have two vertices from each part (see the definition of $\scr K_{m\times n}^{3}$ above). More  precise definitions together with preliminaries  are given in Section \ref{term},  the main result is proved in Section \ref{factorizationcor}, and related open problems are discussed in the last section. 

Amalgamation-detachment technique  was first introduced by Hilton \cite{H2} (who found a new proof for decompositions of complete graphs into Hamiltonian cycles), and was more developed by Hilton and Rodger \cite{HR}. 
Hilton's method was later genealized to arbitrary graphs \cite{bahrodjgtold},  and later to hypergraphs \cite{bahCPC1,Bahhyp1,bahrodjgt13, bahnew16} leading to various extensions of Baranyai's theorem (see for example \cite{bahCPC1, bahCCA1}).  The results of the present paper, mainly relies on those from \cite{bahCPC1} and \cite{Nash87}. For the sake of completeness, here we give a self contained exposition. 
\section{More Terminology and Preliminaries} \label{term} 
Recall that an edge can have multiple copies of the same vertex. For the purpose of this paper, all hypergraphs (except when we use the term graph) are 3-uniform, so an edge is always of one of the  forms  $\{u,u,u\}, \{u,u,v\}$, and $\{u,v,w\}$ which we will abbreviate to   $\{u^3\}, \{u^2,v\}$, and  $\{u,v,w\}$, respectively. In a hypergraph $\mathcal G$,  $\mult_{\mathcal G}(.)$ denotes the multiplicity;  for example $\mult_{\mathcal G}(u^3)$ is the multiplicity of an edge of the form $\{u^3\}$. Similarly, for a graph $G$, $\mult(u,v)$ is the multiplicity of the edge $\{u,v\}$. A \textit{k-edge-coloring} of a hypergraph $\mathcal G$  is a mapping $K:E(\mathcal G)\rightarrow [k]$, and the sub-hypergraph of $\mathcal G$ induced by  color  $i$ is denoted by $\mathcal G(i)$. Whenever it is not ambiguous, we drop the subscripts, and also we abbreviate $d_{\mathcal{G}(i)}(u)$ to $d_i(u)$, $\mult_{\mathcal G(i)}(u^3)$ to $\mult_{i}(u^3)$, etc.. 

Factorizations of the complete graph, $K_n$, is studied in a very general form in \cite{MJohnson07, Johnstone00}, however for the purpose of this paper, a $\lambda$-fold version  is needed:  
\begin{theorem}\textup{(Bahmanian, Rodger \cite[Theorem 2.3]{bahrodsurvey1})}\label{bahrodsurvey1thmfac}
$\lambda K_n$ is $(r_1,\dots,r_k)$-factorable if and only if $r_i n$ is even for $i\in[ k]$ and $\sum_{i=1}^k r_i=\lambda (n-1)$.
\end{theorem}
Let $K_{n}^{*}$ denote the 3-uniform hypergraph with $n$ vertices in which $\mult(u^2, v)=1$, and $\mult(u^3)=\mult(u, v, w)=0$ for  distinct vertices $u,v,w$.  A (3-uniform) hypergraph $\mathcal G=(V,E)$ is {\it $n$-partite},  if there exists a partition $\{V_1,\dots,V_n\}$ of $V$ such that for every $e\in E$, $|e\cap V_i|=1, |e\cap V_j|=2$ for some $i,j\in [n]$ with $i\neq j$.  For example, both $K_n^*$ and $\scr K_{m\times n}^{3}$ are $n$-partite. We need another simple but crucial lemma:
\begin{lemma} If $r_in$ is even for $i\in [k]$, and $\sum_{i=1}^k r_i= \lambda (n-1)$, then $\lambda K_{n}^{*}$ is $(3r_1,\ldots,3r_k)$-factorable. \label{3rfaclemma}
\end{lemma}
\begin{proof}
Let $G=\lambda K_n$ with vertex set $V$. By Theorem \ref{bahrodsurvey1thmfac}, $G$ is $(r_1,\ldots,r_k)$-factorable. Using this factorization, we  obtain a $k$-edge-coloring for $G$ such that $d_{G(i)}(v)=r_i$ for every $v\in V$ and every color $i\in [k]$. Now we form a $k$-edge-colored hypergraph $\mathcal H$ with vertex set $V$ such that $\mult_{\mathcal H(i)}(u^2,v)=\mult_{G(i)}(u,v)$ for every pair of distinct vertices $u,v\in V$, and  each color $i\in [k]$. It is easy to see that $\mathcal H\cong\lambda K_{n}^{*}$ and $d_{\mathcal H(i)}(v)=3r_i$ for every $v\in V$ and every color $i\in [k]$. Thus we obtain a $(3r_1,\ldots,3r_k)$-factorization for $\lambda K_{n}^{*}$. 
\end{proof}

If the multiplicity of a vertex $\alpha$ in an edge $e$ is $p$,  we say that $\alpha$ is {\it incident} with $p$ distinct  {\it hinges}, say $h_1(\alpha,e),\dots,h_p(\alpha,e)$, and we also say that $e$ is {\it incident} with  $h_1(\alpha,e),\dots,h_p(\alpha,e)$. The set of all hinges  in $\mathcal G$ incident with $\alpha$ is denoted by $H_{\mathcal G}(\alpha)$; so $|H_{\mathcal G}(\alpha)|$ is in fact  the degree of $\alpha$. 

Intuitively speaking, an {\it $\alpha$-detachment} of a hypergraph $\mathcal G$ is a hypergraph obtained by splitting a vertex $\alpha$ into one or more vertices and sharing the incident hinges and edges  among the subvertices. That is, in an $\alpha$-detachment $\mathcal G'$ of $\mathcal G$ in which we split $\alpha$ into $\alpha$ and $\beta$,  an edge of the form $\{\alpha^p,u_1,\dots,u_z\}$ in $\mathcal G$ will be of the form $\{\alpha^{p-i},\beta^{i},u_1,\dots,u_z\}$ in $\mathcal G'$ for some $i$, $0\leq i\leq p$. Note that a hypergraph and its detachments have the same hinges. Whenever it is not ambiguous, we use $d'$, $\mult'$, etc. for degree, multiplicity and other hypergraph parameters in $\mathcal G'$.  

Let us fix a vertex $\alpha$ of a $k$-edge-colored hypergraph $\mathcal G=(V,E)$. For $i\in [k]$, let $H_i(\alpha)$ be the set of hinges each of which is incident with both $\alpha$ and an edge of color $i$ (so $d_i(\alpha)=|H_i(\alpha)|$). For any edge $e\in E$, let $H^e(\alpha)$ be the collection of hinges incident with both $\alpha$ and  $e$. Clearly, if $e$ is of color $i$, then $H^e(\alpha)\subset H_i(\alpha)$.

A family $\scr A$ of sets is \textit{laminar} if, for every pair $A, B$ of sets belonging to $\scr A$, either $A\subset B$, or $B\subset A$, or $A\cap B=\varnothing$. We shall present two lemmas, both of which follow immediately from definitions. 
\begin{lemma}\label{lamAlem}
Let $\scr A  =    \{H_{1}(\alpha),\ldots,H_{k}(\alpha)\}     \cup  \{ H^e(\alpha) : e  \in E\}$. Then $\scr A$ is a laminar family of subsets of $H(\alpha)$.
\end{lemma}
For each $p\in \{1,2\}$, and each $U \subset V\backslash\{\alpha\}$, let $H(\alpha^p, U)$ be the set of hinges each of which is incident with both $\alpha$ and an edge of the form $\{\alpha^p\}\cup U$ in $\mathcal G$ (so $|H(\alpha^p, U)|=p\mult(\{\alpha^p,U\}$).
\begin{lemma}\label{lamBlem}
Let $\scr B=\{H(\alpha^p, U): p\in \{1,2\}, U \subset V\backslash\{\alpha\}\}$. Then $\scr B$ is a laminar family of disjoint subsets of $H(\alpha)$.
\end{lemma}
If $x, y$ are real numbers, then $\lfloor x \rfloor$ and $\lceil x \rceil$ denote the integers such that $x-1<\lfloor x \rfloor \leq x \leq \lceil x \rceil < x+1$, and $x\approx y$ means $\lfloor y \rfloor \leq x\leq \lceil y \rceil$. We need the following powerful lemma:
\begin{lemma}\textup{(Nash-Williams \cite[Lemma 2]{Nash87})}\label{laminarlem}
If $\scr A, \scr B$ are two laminar families of subsets of a finite set $S$, and $n\in \mathbb{N}$, then there exist a subset $A$ of $S$ such that 
\begin{eqnarray*} 
 |A\cap P|\approx |P|/n \mbox { for every } P\in \scr A \cup \scr B. 
\end{eqnarray*} 
\end{lemma}

\section{Proofs}\label{factorizationcor}
Notice that $\lambda\scr K_{m\times n}^{3}$ is a $3\lambda(n-1)\binom{m}{2}$-regular hypergraph with $nm$ vertices and $2\lambda m\binom{n}{2}\binom{m}{2}$ edges. To prove Theorem \ref{type2suff}, we prove the following seemingly stronger result.
\begin{theorem} \label{moregenres}
 Let $3\divides r_im$ and $2\divides r_imn$ for $i\in [k]$, and $\sum_{i=1}^k r_i= 3\lambda(n-1)\binom{m}{2}$. Then for all $\ell=n,n+1,\dots,mn$ there exists a $k$-edge-colored $\ell$-vertex $n$-partite hypergraph $\mathcal G=(V,E)$  and a function $g:V\rightarrow \mathbb N$ such that the following conditions are satisfied:
\begin{enumerate}
\item [(C1)] $\sum_{v\in W} g(v)=m$ for each part $W$ of $\mathcal G$;
\item [(C2)] $\mult (u^2,v)=\lambda \binom{g(u)}{2} g(v)$ for each pair of vertices $u,v$ from different parts of $\mathcal G$;
\item [(C3)] $\mult (u,v,w)=\lambda g(u) g(v)g(w)$ for each pair of distinct vertices $u,w$ from the same part, and $v$ from a different part of $\mathcal G$;
\item [(C4)] $d_{i}(u)=r_i g(u)$ for each color $i\in [k]$ and each $u\in V$.
\end{enumerate}
\end{theorem} 
\begin{remark}\textup{ It is implicitly understood that every other type of edge in $\mathcal G$ is of multiplicity 0.
}\end{remark}
Before we prove Theorem \ref{moregenres}, we show how  Theorem \ref{type2suff} is implied by Theorem \ref{moregenres}. 
 \\
 
\noindent {\it \bf Proof of Theorem \ref{type2suff}.}
It is enough to take $\ell=mn$ in Theorem \ref{moregenres}. Then there exists an $n$-partite hypergraph $\mathcal G=(V,E)$ of order $mn$ and  a function $g:V\rightarrow \mathbb N$ such that by (C1)  $\sum_{v\in W} g(v)=m$ for each part $W$ of $\mathcal G$. This implies that $g(v)=1$ for each $v\in V$ and that each part of $\mathcal G$ has $m$ vertices. By (C2),  $\mult_{\mathcal G} (u^2,v)=\lambda \binom{1}{2}(1)=0$ for each pair of vertices $u,v$ from different parts of $\mathcal G$, and by (C3), $\mult_{\mathcal G} (u,v,w)=\lambda$ for each pair of vertices $u,v$ from the same part and $w$ from a different part of $\mathcal G$. This implies that $\mathcal G \cong \lambda \scr K_{m\times n}^{3}$.
Finally, by (C4),  $\mathcal G$ admits a $k$-edge-coloring such that $d_{\mathcal G(i)}(v)=r_i$ for each color $i\in [ k]$. This completes the proof.
\qed

The idea of the proof of Theorem \ref{moregenres} is that each vertex $\alpha$ will be split into $g(\alpha)$ vertices and that this will be done by ``splitting off" single  vertices one at a time. 
\\

\noindent {\it \bf Proof of Theorem \ref{moregenres}.} We prove the theorem by induction on $\ell$.

First we prove the basis of induction, case $\ell=n$. Let $\mathcal G=(V,E)$ be $\lambda m\binom{m}{2}K_{n}^{*}$ and let $g(v)=m$ for all $v\in V$. Since $\mathcal G$ has $n$ vertices, it is $n$-partite (each vertex being a  partite set). Obviously, $\sum_{v\in W} g(v)=g(v)=m$ for each part $W$ of $\mathcal G$. Also, $\mult(u^2,v)=\lambda m\binom{m}{2}=\lambda \binom{g(u)}{2} g(v)$ for each pair of vertices $u,v$ from distinct parts of $\mathcal G$, so (C2) is satisfied. Since there is only  one vertex in each part, (C3) is trivially satisfied. 

Since for $i\in [k]$, $2\divides \frac{r_imn}{3}$ and $\sum_{i=1}^k \frac{r_im}{3}= \lambda m(n-1)\binom{m}{2}$, by Lemma \ref{3rfaclemma}, $\mathcal G$ is $(mr_1,\ldots,mr_k)$-factorable. Thus, we can find a $k$-edge-coloring for $\mathcal G$ such that  $d_{\mathcal G(j)}(v)=mr_i=r_ig(v)$ for $i\in [k]$, and therefore (C4) is satisfied.

Suppose now that for some $\ell \in \{ n, n+1,  \ldots, mn-1 \}$, there exists a $k$-edge-colored $n$-partite hypergraph $\mathcal G=(V,E)$ of order $\ell$ and a function $g:V \rightarrow {\mathbb N}$ satisfying properties (C1)--(C4) from the statement of the theorem. We shall now construct an $n$-partite  hypergraph $\mathcal G'$ of order $\ell+1$ and a function $g':V(\mathcal G') \rightarrow {\mathbb N}$ satisfying  (C1)--(C4). 

Since $\ell<mn$, $\mathcal G$ is $n$-partite and (C1) holds for $\mathcal G$, there exists a vertex $\alpha$ of $\mathcal G$ with $g(\alpha)>1$. The graph $\mathcal G'$ will be constructed as an $\alpha$-detachment of $\mathcal G$ with the help of laminar families  $$\scr A:=    \{H_{1}(\alpha),\ldots,H_{k}(\alpha)\}     \cup  \{ H^e(\alpha) : e  \in E\}$$ and $$
\scr B:= \{H(\alpha^p, U): p\in \{1,2\}, U \subset V\backslash\{\alpha\}\}.$$
By  Lemma \ref{laminarlem}, there exists a subset $Z$ of $H(\al)$ such that 
\begin{equation}\label{lamapp1'} |Z\cap P|\approx |P|/g(\alpha),  \mbox{ for every  } P\in \mathscr A \cup \scr B. \end{equation}

Let $\mathcal G'=(V',E')$ with $V'=V\cup\{\beta\}$ be the hypergraph obtained from $\mathcal G$ by splitting $\alpha$ into two vertices $\alpha$ and  $\beta$ in such a way that hinges which were incident with $\alpha$ in $ \mathcal G$ become incident in $\mathcal G'$ with $\alpha$ or $\beta$ according to whether they do not or do belong to $Z$, respectively. More precisely,
 \begin{equation}\label{hinge1'} 
 H'(\beta)=Z, \quad H'(\alpha)=H( \alpha)\backslash Z. 
 \end{equation}
So $\mathcal G'$ is an $\alpha$-detachment of $\mathcal{G}$ and  the colors of the edges are preserved. Let $g':V' \rightarrow {\mathbb N}$ so that $g'(\al)=g(\al)-1, g'(\beta)=1$, and $g'(u)=g(u)$ for each $u\in V'\backslash \{\al,\beta\}$. 
It is obvious that  $\mathcal G'$ is of order $\ell+1$, $n$-partite, and $\sum_{v\in W} g'(v)=m$ for each part $W$ of $\mathcal G'$ (the new vertex $\beta$ belongs to the same part of $\mathcal{G'}$ as $\alpha$ belongs to).  Moreover, it is clear that  $\mathcal G'$ satisfies (C2)--(C4) if $\{\alpha,\beta\} \cap \{u,v,w\}=\emptyset$. For the rest of the argument, we will repeatedly use the definitions of $\mathscr A, \mathscr B$,  (\ref{lamapp1'}), and (\ref{hinge1'}). 

For $i\in [k]$ we have 
\begin{eqnarray*}
d'_i(\beta) &=&  |Z\cap H_i(\al)|\approx |H_i(\alpha)|/g(\al)
 =  d_i(\alpha)/g(\alpha)=r_i=r_ig'(\beta),\\
d'_i(\alpha) &=&  d_i(\al)-d'_i(\beta)=r_ig(\alpha)-r_i=r_i(g(\al)-1)=r_ig'(\al),
\end{eqnarray*}
so $\mathcal{G'}$ satisfies (C4).   

Let $u\in V'$ so that $u$ and $\alpha$ (or $\beta$) belong to different parts of $\mathcal{G'}$. We have 
\begin{eqnarray*}
\mult'(\beta,u^2)&=&|Z\cap H(\al,\{u^2\})|\approx |H(\al,\{u^2\})|/g(\al)=\mult(\al,u^2)/g(\al)\\
&=&\lambda \binom{g(u)}{2}=\lambda \binom{g'(u)}{2} g'(\beta),\\
\mult' (\al,u^2)&=&\mult(\al,u^2)-\mult' (\beta,u^2)=\lambda \binom{g(u)}{2} g(\al)-\lambda \binom{g(u)}{2}
=\lambda \binom{g'(u)}{2} g'(\al).
\end{eqnarray*}
Recall that $g(\alpha)\geq 2$, and for every $e\in E$ and $i\in[k]$, $|H^e(\alpha)|\leq 2$, and thus $|Z\cap H^e(\alpha)|\approx |H^e(\alpha)|/g(\al)\leq 1$. This implies that 
\begin{eqnarray*}
\mult'(\beta^2,u)=0=\la \binom{g'(\beta)}{2} g'(u),
\end{eqnarray*}
and so $\mult(\al^2,u)=\mult' (\al^2,u)+\mult' (\al,\beta,u)$. Now we have
\begin{eqnarray*}
\mult' (\al,\beta,u)&=&|Z\cap H(\al^2,\{u\})| \approx |H(\al^2,\{u\})|/g(\al)\\
&=&2\mult(\al^2,u)/g(\al)=\lambda (g(\al)-1)g(u)=\lambda g'(\al)g'(\beta)g'(u),\\
\mult' (\al^2,u)&=&\mult(\al^2,u)-\mult' (\al,\beta,u)=\lambda \binom{g(\al)}{2}g(u) -\lambda (g(\al)-1)g(u)\\
&=&\lambda \binom{g(\al)-1}{2} g(u)=\la \binom{g'(\al)}{2} g'(u).
\end{eqnarray*}
Therefore $\mathcal{G'}$ satisfies (C2).  

Let $u,v\in V'$ so that $u,v$ belong to  different parts of $\mathcal{G'}$,  $u,\alpha$ belong to the same part of $\mathcal{G'}$, and $u\notin\{\alpha,\beta\}$. We have 
\begin{eqnarray*}
\mult'(\beta, u, v)&=&|Z\cap H(\al,\{u,v\})|\approx |H(\al,\{u,v\})|/g(\al)=\mult(\al,u,v)/g(\al)\\
&=&\lambda g(u)g(v)=\lambda g'(\beta)g'(u)g'(v),\\
\mult' (\al,u,v)&=&\mult(\al,u,v)-\mult'(\beta,u,v)=\lambda (g(\al)-1)g(u)g(v)= \lambda g'(\al)g'(u)g'(v).
\end{eqnarray*}
Finally, let $u,v\in V'$ so that $u,v$ belong to  the same part of $\mathcal{G'}$, and $u,\alpha$ belong to different parts of $\mathcal{G'}$, and $u\notin\{\alpha,\beta\}$. By an argument very similar to the one above, we have 
\begin{eqnarray*}
\mult'(u, v,\beta)&=&\lambda g'(u)g'(v)g'(\beta),\\
\mult' (u,v,\alpha)&=& \lambda g'(u)g'(v)g'(\al).
\end{eqnarray*}
Therefore $\mathcal{G'}$ satisfies (C3), and the proof is complete.  
\qed
\section{Final Remarks}\label{remarks}
We define $\scr K_{m_1,\ldots,m_n}^{3}$ similar to $\scr K_{m\times n}^{3}$ with the difference that in $\scr K_{m_1,\ldots,m_n}^{3}$ we allow different parts to have different sizes. It seems reasonable to conjecture that
\begin{conjecture} \label{type2suffconj}
$\lambda \scr K_{m_1,\ldots,m_n}^{3}$ is $(r_1,\ldots,r_k)$-factorable if and only if 
\begin{enumerate}
\item  [\textup {(i)}] $m_i=m_j:=m$ for $i,j\in [n]$,
\item  [\textup {(ii)}] $3\divides r_imn$ for $i\in [k]$, and
\item  [\textup {(iii)}] $\sum_{i=1}^{k} r_i= 3\lambda(n-1)\binom{m}{2}$.  
\end{enumerate}
\end{conjecture}
We prove the necessity as follows. Since $\lambda \scr K_{m\times n}^{3}$ is factorable, it must be regular. Let $u$ and $v$ be two vertices from two different parts, say $p^{th}$ and $q^{th}$ parts respectively. Then we have the following sequence of equivalences:
\begin{align*}
d(u)&=d(v) &\iff \\
\sum\nolimits_{\scriptstyle 1 \leq i  \leq n \hfill \atop  \scriptstyle i \ne p \hfill}  \binom{m_i}{2}+(m_p-1)\sum\nolimits_{\scriptstyle 1 \leq i  \leq n \hfill \atop  \scriptstyle i \ne p \hfill}m_i&=\\
\sum\nolimits_{\scriptstyle 1 \leq i  \leq n \hfill \atop  \scriptstyle i \ne q \hfill}  \binom{m_i}{2}+(m_q-1)\sum\nolimits_{\scriptstyle 1 \leq i  \leq n \hfill \atop  \scriptstyle i \ne q \hfill}m_i &&\iff \\
\binom{m_q}{2}+\sum\nolimits_{\scriptstyle 1 \leq i  \leq n \hfill \atop  \scriptstyle i \ne p,q \hfill}  \binom{m_i}{2}+(m_p-1)(m_q+\sum\nolimits_{\scriptstyle 1 \leq i  \leq n \hfill \atop  \scriptstyle i \ne p,q \hfill}m_i)&= \\
\binom{m_p}{2}+\sum\nolimits_{\scriptstyle 1 \leq i  \leq n \hfill \atop  \scriptstyle i \ne p,q \hfill}  \binom{m_i}{2}+(m_q-1)(m_p+\sum\nolimits_{\scriptstyle 1 \leq i  \leq n \hfill \atop  \scriptstyle i \ne p,q \hfill}m_i) &&\iff& \\
\binom{m_p}{2}-\binom{m_q}{2}+m_pm_q-m_p-m_pm_q+m_q+(m_p-m_q)\sum\nolimits_{\scriptstyle 1 \leq i  \leq n \hfill \atop  \scriptstyle i \ne p,q \hfill}m_i)&=0 &\iff \\
m_p^2-m_q^2-3m_p+3m_q+2(m_p-m_q)\sum\nolimits_{\scriptstyle 1 \leq i  \leq n \hfill \atop  \scriptstyle i \ne p,q \hfill}m_i)&=0 &\iff \\
(m_p-m_q)(m_p+m_q-3+2\sum\nolimits_{\scriptstyle 1 \leq i  \leq n \hfill \atop  \scriptstyle i \ne p,q \hfill}m_i)&=0 & \iff \\
  m_p&=m_q:=m.
\end{align*}
This proves (i). The existence of an $r_i$-factor implies that $3\divides r_imn$ for $i\in [k]$. Since each $r_i$-factor is an $r_i$-regular spanning sub-hypergraph and $\lambda \scr K_{m\times n}^{3}$ is $3\lambda(n-1)\binom{m}{2}$-regular, we must have $\sum_{i=1}^{k} r_i=3\lambda(n-1)\binom{m}{2}$.

In Theorem \ref{type2suff}, we made partial progress toward settling Conjecture \ref{type2suffconj}, however at this point, it is not clear to us whether our approach will work for the remaining cases.

\section{Acknowledgement}
The author is deeply grateful to  Professors Chris Rodger,  Mateja \v Sajna, and the anonymous referee for their constructive comments.

\end{document}